\documentclass[11pt,a4paper,headinclude,footinclude]{amsart}                 
\usepackage[T1]{fontenc}             
\usepackage[utf8]{inputenc}          \usepackage[english]{babel}       
\usepackage{graphicx}                      % 
\usepackage[font=small]{quoting}            %  
\usepackage{caption}  
\usepackage[colorlinks=true,linkcolor=blue]{hyperref}
\usepackage{cite}
\usepackage{verbatim}
\usepackage{color}
\usepackage{picture}
\usepackage{amsmath,amsthm,amsfonts,amssymb}
\usepackage{comment}
\usepackage{hyperref}
\hypersetup{colorlinks,linkcolor={blue},citecolor={blue},urlcolor={red}}

\newtheorem{thm}{Theorem}[section] 
\newtheorem{lem}[thm]{Lemma} 
\newtheorem{cor}[thm]{Corollary} 
\newtheorem{prop}[thm]{Proposition}

\theoremstyle{definition} 
\newtheorem{rem}[thm]{Remark} 
\theoremstyle{remark}

\theoremstyle{definition}

\def\O{\Omega}
\def\e{\epsilon}
\def\S{\Sigma} 
\def\n{\nabla}
\def\P{\partial}
\def\p{\partial}

\def\a{\alpha}

\def\n{\nabla}

\def\O{\Omega}
\def\p{\partial}
\def\e{\epsilon}
\def\a{\alpha}

\def\n{\nabla}
\def\<{\langle}
\def\>{\rangle}

\def\n{\nabla}

\def\RR{\mathbb{R}}

\def\SS{\mathbb{S}}

\def\BB{\mathbb{B}}

\def\O{\Omega}
\def\p{\partial}
\def\e{\epsilon}
\def\a{\alpha}

\def\R{\mathbb{R}}

\def\C{\mathcal{C}}

{\left\{\begin{array}{@{}l@{}}}{\end{array}\right.}
\patchcmd{\abstract}{\scshape\abstractname}{\textbf{\abstractname}}{}{}
\makeatletter 
\def\@makefnmark{} 
\makeatother

\def\riq{relative isoperimetric inequality } \def\riqq {relative isoperimetric inequality}
  \def\iqq{isoperimetric inequality}

\numberwithin{equation}{section}
\numberwithin{exa}{section}
 \numberwithin{defn}{section}

 \makeatletter
 \@namedef{subjclassname@2020}{\textup{2020} Mathematics Subject Classification}
 \makeatother
\begin{document}
\title[The relative isoperimetric inequality for minimal submanifolds]{The relative isoperimetric inequality for minimal submanifolds with free boundary in the Euclidean space}
\author{Lei Liu, Guofang Wang and Liangjun Weng}
\address{School of Mathematics and Statistics \& Hubei Key Laboratory of Mathematical Sciences, Central China Normal University, Wuhan, 430079, P.R. China}
\email{leiliu2020@ccnu.edu.cn}

\address{{Mathematisches Institut, Albert-Ludwigs-Universit\"{a}t Freiburg, Freiburg im Breisgau, 79104, Germany}}
\email{guofang.wang@math.uni-freiburg.de}

\address{School of Mathematical Sciences, Anhui University, Hefei, 230601, P. R. China}

\email{ljweng08@mail.ustc.edu.cn}



\maketitle
\begin{abstract}
In this paper, 
we mainly consider the relative isoperimetric inequalities for minimal submanifolds with free boundary. 
We first generalize ideas of restricted normal cones introduced by  Choe-Ghomi-Ritor\'e in
\cite{CGR06} and obtain an optimal  area estimate for generalized restricted normal cones.  
This area estimate, together with the ABP method of  Cabr\'e in \cite{Cabre2008},  provides a new  proof of the 
relative isoperimetric inequality obtained by Choe-Ghomi-Ritor\'e in \cite{CGR07}.
Furthermore, we  use  this estimate and the idea of Brendle in his recent work \cite{Brendle2019} to obtain a relative isoperimetric inequality for minimal submanifolds with free boundary on a convex support surface in $\mathbb{R}^{n+m}$, which is optimal 
and gives an affirmative answer to an open problem proposed by Choe in \cite{Choe2005}, Open Problem 12.6,  when the 
codimension $m\le 2$. 
\end{abstract}

\section{Introduction} 

Recently problems related the hypersurfaces with free boundary on a given support surface attract more and more mathematicians. A nice example is the work of Fraser-Schoen on  the first Steklov eigenvalues and minimal free boundary surfaces
\cite{FS1, FS2}, which opens many interesting questions. There have been a lot of results on free boundary hypersurfaces. Here we just mention two further classes of results. One is work of Li and Zhou and their colleagues on the Min-max theory for free boundary minimal hypersurfaces \cite{LZ1, LZ2, LZ3};
	another is the mean curvature flows for free boundary hypersurfaces \cite{Stahl1, Stahl2, LS, SWX}.
	The aim of this paper is to establish the optimal relative isoperimetric inequalities for minimal submanifolds with free boundary, which should be a powerful tool in the study of free boundary hypersurfaces.

We start with
the classical isoperimetric inequality, which states
\begin{equation}
 \label{eq_iso1}
 	\frac{|\partial \O|}{|\partial \mathbb{B}^n|}\geq  \bigg(\frac{|\O|}{|\BB^n|}\bigg)^{\frac{n-1}{n}},
 	\end{equation}
for a bounded domain $\O\subset \RR^n$ and equality holds if and only if $\O$ is  a ball. %$=\BB^n$, the unit ball, 
It  plays a fundamental role in mathematics.
Its origin goes back to antiquity, known as  Dido's problem. It is a longstanding open problem if \eqref{eq_iso1} holds for domains
in a minimal submanifold $M^n$ in the Euclidean space $\R^{n+m}$, which goes back at least to Carleman \cite{Car}, see also \cite{Hsiung}.  For domains in an area-minimizing $M$ in $\RR^{n+m}$,  Almgren solved this open problem affirmatively
 in \cite{A}. See also a proof for $n=2$ in\cite{W}.
There have been many results on this problem, especially when $M$ is a minimal surface. Here we just  
refer to a survey  of Choe \cite{Choe2005} and references therein. Till the recent work of  Brendle, this problem is widely open. In \cite{Brendle2019} Brendle
achieved a breakthrough on this problem  and proved
 \begin{thm}[\textbf{Brendle \cite{Brendle2019}}]\label{thm1.1}
 	Let $M\subset\mathbb{R}^{n+m}$  $(m\ge 1)$ be a compact $n$-dimensional {submanifold} with boundary $\p M$, then
 	\begin{equation}\label{eq_0}
 	\frac{|\p M|+\int_M|H|dv}{|\partial \mathbb{B}^n|}\geq b_{n,m} \left(\frac{\mbox{}|M|}{\mbox{}|\mathbb{B}^n|}\right)^{\frac{n-1}{n}},
 	\end{equation}
 	where $b_{n,m}$  is defined by 
 	\begin{equation}\label{eq_0.1}
 	 b_{n,m} =\begin{cases} \left(\frac{(n+m)|\mathbb{B}^{n+m}|}{m |\mathbb{B}^n||\mathbb{B}^m|}\right)^{\frac{1}{n}}, & \hbox{ if } m> 2, \\
 	           1, & \hbox{ if } m=1 \hbox{ or } 2,
 	          \end{cases}
 	\end{equation}
with equality for $m\le 2$ if and only if  $M$ is a  round ball. Here $H$ is the mean curvature of $M$, $|\p M|$ and $|M|$  are the area and the volume of $\p M$ and $M$ respectively.
 \end{thm}
 Theorem \ref{thm1.1} 
implies that the above longstanding open  problem has an affirmative answer if the codimension is not bigger than $2$,
while gives an explicit constant $b_{n,m}$ for the general case.
Inequality \eqref{eq_0}
 implies a Michael-Simon and  Allard inequality  \cite{MS}, \cite{Allard}  with an optimal positive constant, at least in the case of codimension $m\le 2$. 
 See \cite{Brendle2019}. For previous work see  \cite{Castillon}. 

In this paper, we are interested in  the so-called the relative isoperimetric inequality for 
$n$-dimensional minimal submanifolds in $\mathbb{R}^{n+m}$.  Set $N:=n+m$.
Let $\mathcal{C}\subset \mathbb{R}^{N}$ be an open convex body in $\mathbb{R}^{N}$ with a smooth boundary
 $S=\p \C$.  Let $M \subset \RR^N$ be an $n$-dimensional submanifold with codimension $m=N-n$. Its boundary $\p M$
 consists of two smooth pieces $\Sigma$ and $\Gamma$, where $\Gamma\subset \p \C$ and $\Sigma 
 \subset 
 \R^n\backslash \C$. Denote their common boundary by $\p\Sigma$, which may be empty.
 %Both $\Sigma$ and $\Gamma$ may be or may not be conn
 Let $\nu$ be the outer unit normal vector field of $\p M \subset M$
 and $\nu_S$   the outer unit normal vector field of $S$. We assume that  $\nu=-\nu_S$ along $\Gamma$, i.e. $M$ intersects $S$ from outside orthogonally, and  call such a submanifold 
 a {\it partially free boundary submanifold} with free boundary $\Gamma$ and relative boundary (or fixed, or Dirichlet boundary) $\Sigma$. 
 When $\S=\emptyset$, we call $M$ a {\it  free boundary submanifold}.
  In some literature, $S$ is also called a support
 hypersurface. 
 The relative isoperimetric inequality concerns the relation between the area of the relative boundary $\Sigma$,  $| \Sigma|$, and the volume
 of $M$,  $|M|$.

We prove the following relative isoperimetric inequality for submanifolds in $\mathbb{R}^{n+m}$, which is a relative version of the result of  Brendle \cite{Brendle2019}.
 
 \begin{thm}\label{thm1.2}
 	Let $M^n\subset \R^{n+m}$ $(m\ge 1)$ be a partially free  boundary   submanifold 
with relative boundary  $\Sigma$ and free boundary $\Gamma$ on a convex support hypersurface $S$. We have
\begin{equation}\label{riq}
 	\frac{|\S|+\int_M|H|dv}{|\partial \mathbb{B}^n|}\geq 
 	\left(\frac 12 \right)^{\frac{1}{n}} b_{n,m}
 	\left(\frac{\mbox{}|M|}{\mbox{}|\mathbb{B}^n|}\right)^{\frac{n-1}{n}},
 	\end{equation}
 	where $b_{n,m}$ is defined in \eqref{eq_0.1} and $H$ is the mean curvature vector of $M$.
Moreover, when $m\le 2$,  equality holds if and only if 
 $M$ is a flat half $n$-ball.
 \end{thm}

  Inequality \eqref{riq} also implies  a relative version of  Michael-Simon and  Allard inequality. See  Theorem \ref{coro2} at the end of Section \ref{sec3}.
 
%Our results can be viewed as a generalization of the  relative isoperimetric inequality for domains in the Euclidean space proved finally by

 As a Corollary,  we obtain  the optimal  relative isoperimetric inequality for minimal submanifolds
 in the Euclidean space, provided that the codimension $m:=N-n \le 2$.
 Namely we solve the open problem  which  was proposed  by  Choe, Open problem 12.6  in \cite{Choe2005}, if the  codimension is not bigger than  $2$.
 \begin{cor}\label{coro1}
 Let $M^n\subset \R^{n+m}$ $(m\le 2)$ be a  partially free  boundary minimal   submanifold 
with relative boundary  $\Sigma$ and free boundary $\Gamma$ on a convex support hypersurface $S$. We have
	\begin{equation}\label{riq2}
 	\frac{|\S|}{|\partial \mathbb{B}^n|}\geq \left(\frac 12\right)^{\frac{1}{n}} \left(\frac{\mbox{}|M|}{\mbox{}|\mathbb{B}^n|}\right)^{\frac{n-1}{n}},
 	\end{equation}
moreover,  equality holds if and only if 
 $M$ is a flat half $n$-ball.
 \end{cor}
 Inequality \eqref{riq2} is equivalent to
 \[
 \frac{|\S|}{| \SS_+^{n-1}|}\geq \left(\frac{\mbox{}|M|}{\mbox{}|\mathbb{B}_+^n|}\right)^{\frac{n-1}{n}},
 \]
 where $\BB^n_+$ is the unit half ball and $\SS^{n-1}_+$ is the unit half sphere.  We emphasize that  both   conditions,  the orthogonality of the intersection between $M$ and   $S$  and  the convexity of $S$ are
  necessary, since it is 
 easy to find a counterexample if one of these two  conditions is missing.
  There has been a lot of work on the relative isoperimetric inequality on minimal submanifolds,
 especially on minimal surfaces by Choe and his school. See again the nice survey of Choe \cite {Choe2005} and references therein. An Almgren type result was proved recently
 by Krummel \cite {Kru2017}, namely, the \riq holds when $M$ is area-minimizing with partially free boundary on a convex hypersurface,
 following closely the method given by Almgren \cite{A}.

 When $M=\O$ is a bounded domain in $\RR^n$, i.e., $N=n$,  one can view $\O$ as a minimal submanifold in $\RR^n$ with codimension $0$.
  In this case, Corollary \ref{coro1} is the relative isoperimetric inequality 
 proved by Choe, Ghomi and Ritor\'e in 2007. 
 
 \begin{thm}[\textbf{Choe-Ghomi-Ritor\'e \cite{CGR07}}]\label{thm1.4} 
 Let $\Omega=M \subset  \RR^ n$ be a bounded domain outside a convex body  $\C$ with its 
boundary $\P \Omega$ consisting of two smooth pieces $\Sigma$ and $\Gamma$, 
where $\Gamma\subset \p \C$ 
%\textcolor{red}
{and $\Sigma \subset \R^n\backslash {\mathcal C}$}. 
 % Both $\Sigma$ and $\Gamma$ can be non-connected.
  Then
  \begin{equation} \label{eq_iso2}
 	\frac{|\S|}{|\partial \mathbb{B}^n|}\geq  \left(\frac 12 \right)^{\frac 1 n}  \left(\frac{|\O|}{|\BB^n|}\right)^{\frac{n-1}{n}},
 	\end{equation}
moreover,	equality  holds  if and only if  $\O$ is a flat  half $n$-ball.
 \end{thm}

 It is  this relative isoperimetric inequality and other results for minimal surfaces obtained by  Choe and his colleagues 
that motivated  Choe to propose the above open problem in \cite{Choe2005} in 2005.

There have been a lot of proofs of the classical isoperimetric inequality \eqref{eq_iso1}.  Here we just mention one proof for smooth domains, which is important for our paper.
In  \cite{Cabre2008} Cabr\'e provided a proof by using ideas of the Alexandrov-Bakelman-Pucci maximum principle. A similar  idea was used by Trudinger in \cite{T}. See also further applications in 
 \cite{Cabre17}.
 Brendle's  method to prove Theorem \ref{thm1.1}   is a clever extension of the ABP proof of Cabr\'e. The first interesting question is: 
 whether there is an ABP proof for the \riq of Choe-Ghomi-Ritor\'e \cite{CGR07}, namely, Theorem \ref{thm1.4}?  In this paper we 
 first provide an ABP proof for the \riqq, by modifying  interesting ideas, the restricted normal cones, 
 given in another paper of Choe-Ghomi-Ritor\'e \cite{CGR06}. The original proof of the \riq in  \cite{CGR07}  relies crucially on   
 \cite{CGR06}. 
 The key  is Proposition \ref{prop1.4} below,  in which we obtain an optimal area estimate for generalized normal cones defined in the next section.
 This area estimate
 generalizes a key result  proved by Choe-Ghomi-Ritor\'e \cite{CGR06}. This is our main own contribution.
 With this optimal area estimate
 we  use the ABP technique in \cite{Cabre2008} to provide a new proof of Theorem \ref{thm1.4}. % or precisely Proposition \ref{prop1.4},  
Then we use it and the tricks given by Brendle to generalize the results in \cite{Brendle2019} to the relative case.
A boundary version of 
 Michael-Simon and  Allard inequality  \cite{MS}, \cite{Allard} follows now easily. See Theorem 4.4.  The result of Brendle for the 
 logarithmic Sobolev inequality in  \cite{Brendle2019b} can also be generalized to obtain a relative version.
 
The isoperimetric inequalities,  the Michael-Simon and  Allard inequality,  are very useful in differential
 geometry and geometric analysis, especially in the study of minimal submanifolds and curvature flows for closed submanifolds. We believe that 
 our relative inequalities are also very
 useful for the related problems with boundary, especially curvature flows of submanifolds (or hypersurfaces) with free boundary, on which 
 there has been recently a lot of work as partially mentioned above.

 \
 
\textit{ The paper is organized as follows.} In Section \ref{sec0}, we first prove the key Proposition, Proposition \ref{prop1.4},
and then provide an ABP proof for the \riq of domain in $\mathbb{R}^n$, namely Theorem \ref{thm1.4} in Section \ref{sec2}.
We   prove Theorem \ref{thm1.2} and the relative  Michael-Simon and  Allard inequality, Theorem \ref{coro2},
 in Section \ref{sec3}.

 \section{A generalized cone}\label{sec0}

Let us first introduce the (unit) {\it normal cone}, which is a standard concept. See for instance \cite{CGR06}.  For any subset $X \subset \R^N$ and any point $p\in \R^N$, the
(unit) normal cone of $X$ at $p$ is defined by
\[N_pX:=\{ \xi \in \SS^{N-1} \,|\, \< x-p,\xi \> \le 0, \quad \forall \, x\in X\}.\]
Set
\[
 NX:=\cup_{p\in X} N_pX.
\] If we have further a map $\sigma:X\to\SS^{N-1}$, we define the {\it restricted normal cone} of $X$ at $p\in X$ (with respect to $\sigma$) as in \cite{CGR06}
\[
 N_pX\slash \sigma:= N_pX \cap {H_{\sigma(p)}},
\]where $H_\eta$ ($\eta\in \SS^{N-1}$) is the half space defined by
\[
H_\eta:=\{y\in \R^N \,|\, \<y, \eta\> \geq 0\},
\]
and set
\[
 NX\slash \sigma:= \cup_{p\in X}  N_pX\slash \sigma.
\]

The following interesting Proposition was proved in \cite[Proposition 5.3]{CGR06}.
\begin{prop}[\textbf{Choe-Ghomi-Ritor\'e  \cite{CGR06}}]\label{prop1.1}
Let $X\subset \R^N$ be a compact set which is disjoint from the relative interior of its convex hull. 
Suppose there exists a continuous mapping $\sigma : X\to \SS^{N-1}$ such that $\sigma(p) \in N_pX$ for all $p\in X.$ Then,
\begin{equation}
 \label{eq1.3}
 |NX\slash \sigma| \ge \frac 12  |\SS^{N-1}|.
\end{equation}
\end{prop}

\begin{rem} 
 \label{rmk2.2} For the later use, we need to consider the normal cone and the restricted normal cone of 
 length $\rho\in (0, \infty)$. The normal cone of
 length $\rho\in (0,\infty)$ of $X$ at $p$  is defined
 by
 \[
  N^\rho_pX:=\{ \xi \in \SS^{N-1}(\rho) \,|\, \< x-p,\xi \> \le 0, \quad \forall \, x\in X\},
 \]
where $\SS^{N-1}(\rho)$ is the sphere of radius $\rho$ centered at the origin.
One can similarly define $N^\rho X$ and $N^\rho X\slash \sigma$
for a function  $\sigma(p) \in N_pX$. Due to the scaling invariance of the condition $\< x-p,\xi \> \le 0$ for $\xi$, it
is trivial to see that $N^\rho X=\rho NX.$ Hence,
under the same assumptions as in Proposition \ref{prop1.1} we have
\begin{equation}\label{eq1.4}
 |N^\rho X\slash \sigma|\ge \frac 12 |\SS^{N-1}(\rho)|=\frac 12 \rho^{N-1}|\SS^{N-1}|.
\end{equation}
The scaling invariance is clearly not true for the generalized normal cone, which we will now define.
\end{rem}

Now we generalize the  concept of the normal cone and the restricted normal cone to our case as follows. 
For any subset $X \subset \R^N$, any point $p\in \R^N$ and any function $u:\R^N \to \R$, we introduce the
 {\it generalized normal cone} of length $\rho$ of $X$ at $p$ with respect to $u$ defined by
\[N^{u,\rho}_pX:=\{ \xi \in \SS^{N-1} (\rho)\,|\, \< x-p,\xi \> \le u(x)-u(p), \quad \forall \, x\in X\}\]
and set
\[
 N^{u,\rho}X:=\cup_{p\in X} N^{u,\rho}_pX.
\]% For our application, we need to consider all $\rho \in (0,1)$. 
Here, for the simplicity of notation, we omit the superscript $\rho$ if there is no confusion.
If we have further a map $\sigma:X\to\SS^{N-1}$, we define the {\it generalized restricted normal cone} of $X$ at $p\in X$ with respect to $u$ 
\[
 N^u_pX\slash \sigma:= N^u_pX \cap {H_{\sigma(p)}}
\]
and set
\[
 N^uX\slash \sigma:= \cup_{p\in X}  N^u_pX\slash \sigma.
\]

When $u$ is a constant function, then  both definitions are certainly the same. But when $u$ is not a constant function, 
there are at least  two big differences, which prevent us to directly use the results in \cite{CGR06} to prove an analogous inequality to \eqref{eq1.3} or \eqref{eq1.4}.
The first difference is that
the condition 
\[\< x-p,\xi \> \le u(x)-u(p),\]
is not scaling invariant. This causes that
$N^u_pX$ might be not spherical convex, while $N_pX$ is. The latter is crucial  for the validity  of \eqref{eq1.3}.
The second  difference, which is also crucial in the later applications, is that  
$\sigma(p)=\nu (p) \in N_p\Gamma$ is true, when $\Gamma=X$ is a set lying on the convex hypersurface, 
but $\sigma(p)=\nu (p)\in N^u_p\Gamma$ is in general not true.

To overcome these difficulties, instead of $X$, we consider the graph of $u$
\[
\tilde  X:=\{ (x, u(x)) \,|\, x\in X\}
,\]
in $\R^N\times \R=\R^{N+1}$. For $x\in X$, we set $\tilde x= (x, u(x)) \in \tilde X$. It is important to remark that if $X$ is disjoint from the relative interior
of its convex hull, so is  $\tilde X$.
We consider the  normal cone $N^{\tilde \rho} \tilde X=\cup_{\tilde  p\in \tilde X}N^{\tilde \rho} _{\tilde p}\tilde X$ of length $\tilde \rho= \sqrt{1+\rho^2}$ 
by viewing $\tilde X$ as a subset in $\R^{N+1}$,
i.e.
\[
 N^{\tilde \rho}_{\tilde p}\tilde X=\{
 \bar \xi:=(\xi, \xi_{N+1})\in \SS^{N}(\tilde{\rho}) \, |\, \<\tilde  x-\tilde  p, \bar \xi \> \le 0, \quad \forall \tilde x\in \tilde X
 \}.
\]
In this way we embed a generalized normal cone $N^u_pX$ w.r.t to a function $u$  into a normal cone $N^{\tilde \rho}_{\tilde p}\tilde X$, with $\tilde \rho:=\sqrt{1+\rho^2}$. 
For the simplicity of  notation we also omit the superscript $\tilde \rho$, if there is no confusion. 
We have the following simple observation.
\begin{lem} \label{lem2} For any $\rho \in (0,  \infty)$, set  $\tilde \rho=\sqrt{1+\rho^2}$. We have
 \begin{equation}
 \label{embedding}
\xi\in  N^{u,\rho}_pX \Longleftrightarrow 
(\xi, -1) \in  N_{\tilde p}^{\tilde \rho}\tilde X.
\end{equation}
As a result, we can identify $ N^u_pX$ with   $N_{\tilde p}\tilde X \cap \{\xi_{N+1}=-1\}.$
\end{lem}
\begin{proof}
 The Lemma follows trivially from the fact that the statement
 \[
   \< x-p,\xi \> \le u(x)-u(p), \quad \text{ for all } x\in X,
 \] 
is equivalent to the statement
\[
 \< \tilde x-\tilde p,\bar \xi \> \le 0, \quad \text{ for all } \tilde{x}\in \tilde{X} \text{  with }  \bar \xi :=(\xi, -1).
\]  Here, also as above, we use the notation that $\tilde x:=(x, u(x))$ and $\tilde p:=(p, u(p))$.
\end{proof}

Now we state our generalization of Proposition \ref{prop1.1}.

\begin{prop}\label{prop1.4}
Let $X\subset \R^N$ be a compact set that is disjoint from the relative interior of its convex hull and $u:X\to \RR$ a continuous function. 
Suppose there exists a continuous mapping $\sigma : X\to \SS^{N-1}$ such that 
$\bar\sigma := \tilde \rho \cdot (\sigma(p), 0) \in N_{\tilde p} \tilde X\subset \SS^N(\tilde \rho)$ for all $p\in X.$ Then,
\begin{equation}
 \label{eq1.9}
 |N^uX\slash \sigma| \ge \frac 12  |\SS^{N-1}(\rho)|.
\end{equation}
\end{prop}

\begin{rem}
 The assumption on the map $\sigma$ is crucial. In the later applications $\nu(p)\in N^u\Gamma$ is in general not true, however, it is
 easy to see that $\tilde \rho (\nu(p), 0 ) \in N_{\tilde p} \tilde \Gamma$ is true, and hence Proposition \ref{prop1.4} can be applied.
\end{rem}

To prove Proposition \ref{prop1.4}, the following observation is crucial.

\begin{lem}
 \label{lem1.6}
 Under the same assumptions as in  Proposition \ref{prop1.4}, there holds
 \begin{equation}\label{eq1.10}
    |N^u_pX\slash \sigma | \ge \frac 12 |N^u_pX|,
 \end{equation}
 for any $p\in X$. 
\end{lem}
\begin{proof}
Without loss of generality, we may assume that $\sigma(p)=(1, 0,\cdots, 0) \in\mathbb{S}^{N-1}$ 
and $\bar \sigma=\tilde \rho (1, 0, \cdots, 0, 0)\in N_{\tilde p} \tilde X.$ 
By Lemma 4.1 in \cite{CGR06}, we know that the normal cone  $N_{\tilde p} \tilde X$ is   a convex  spherical set on $\SS^N(\tilde \rho)$. 
By Lemma \ref{lem2}  for any point $\xi \in N^u_pX $, we have $(\xi, -1) \in N_{\tilde p} \tilde X \subset  \SS^N(\tilde \rho)$.  %seen as a point in $N^{\tilde \rho}_{\tilde p}\tilde X \cap \{\xi_{N+1}=-1\}$ by Lemma \ref{lem2},
Hence the geodesic segment connecting $\bar \sigma$ and $(\xi, -1)$ lies entirely in  $N_{\tilde p}\tilde X$. 
Set $A:= N^u_pX\slash \sigma = \{\xi\in N^u_pX\, |\, \xi_1\ge 0 \} $ and $B= N^u_pX \setminus A=
\{\xi\in N^u_pX\, |\, \xi_1< 0  \}$. Let $\tilde A:= A\times \{-1\}$  and $\tilde B:= B\times \{-1\}$ and define $\tilde B'$ be the reflection
of $B$ with respect to the hyperplane $\{\xi_1=0\}$ in $\mathbb{R}^{N+1}$, i.e., $\tilde B' =\{(-\xi_1, \xi_2, \cdots, \xi_N, -1) \,|\, \tilde{\xi}:=(\xi_1, \xi_2, \cdots, \xi_N, -1)\in \tilde B\}$.
We \textbf{claim} that $\tilde B'\subset \tilde A$. Then the Lemma follows from this claim.  In fact, for any $\tilde \xi:=(\xi_1, \xi_2, \cdots, \xi_N, -1)\in \tilde B$, 
by definition $\xi_1< 0$. From the above discussion, we know that the geodesic segment   on $\SS^{N}(\tilde \rho)$
connecting $\bar \sigma $ and $\tilde{\xi}$  lies on $N_{\tilde p}\tilde X$. One can see easily that this segment goes through 
the point $(-\xi_1, \xi_2, \cdots, \xi_N, -1)$ with $-\xi_1\ge 0$. It is clear that it lies in $\tilde A$, and hence  $\tilde B' \subset \tilde A$. Hence we have
\[|A|=|\tilde A| \ge |\tilde B'| = |\tilde B|=|B|.\]
The Lemma follows.
\end{proof}
Now we follow the approach given in \cite{CGR06} to show first that Proposition \ref{prop1.4} is true for a finite set.
\begin{lem} \label{lem3}
  Proposition \ref{prop1.4} is true, if  $X=\{x_1, x_2, \cdots, x_k\}$ is a finite set.
\end{lem}

\begin{proof} First of all, it is easy to see that
\begin{equation}\label{eq2.1}
  N^uX= \SS^{N-1}(\rho).
\end{equation}
In fact, by definition, we have $N^uX\subset \SS^{N-1}(\rho).$
We only need to check   $\SS^{N-1}(\rho) \subset N^uX.$
Note that for any fixed  $\xi \in \mathbb{S}^{N-1}(\rho)$, the function $u(y) -\langle y,\xi\rangle$ 
	attains its minimum at a certain point $p\in X$, for $X$ is a finite set. Namely
	\[u(p) -\langle p,\xi\rangle \le  u(y) -\langle y,\xi\rangle \quad \forall y\in X,\]
	which is equivalent to  $\xi \in N^u_{p} X$, and hence  $\SS^{N-1}(\rho) \subset N^uX.$ Namely, \eqref{eq2.1} holds.

Now we \textbf{claim} that
$ {\mbox{int}(N^u_{x_i}X)\cap \mbox{int}(N^u_{x_j}X)=\emptyset}$ for any $i\neq j$. 
If not,  we may assume that there is an open set  of $U$ such that $U\subset \mbox{int}(N^u_{x_1}X)\cap \mbox{int}(N^u_{x_2}X) $. 
For each $\xi\in U$ we have by definition
\[ \langle y-x_i,\xi \rangle \le  u(y)-u(x_i)\quad \text{ for } y\in X.\]
For $i=1$, by choosing $y=x_2$ in the above inequality we have
\[ \langle x_2-x_1,\xi \rangle \le  u(x_2)-u(x_1).\]
For $i=2$, we choose $y=x_1$ and obtain another inequality. Both together give us
\[ \langle x_2-x_1,\xi \rangle = u(x_2)-u(x_1),\]
which is true for a non-empty open set $U$. It is clear that this is impossible.

From the \textbf{claim}, the previous Lemma and \eqref{eq2.1}  we can complete the proof of the Lemma
	%Since each $N_{x_j}^\rho\Sigma^u\subset  \mathbb{S}^{n-1}(\rho)$ has the area estimate, using Lemma \ref{area estimate for half}, it yields that		
		\begin{align*}
		|N^uX/\sigma | &= 	| \bigcup_{j=1}^k N^u_{x_j}X/\sigma |
		%\sum_{j=1}^J|( N_{x_j}^\rho\Sigma^u\cap \mathcal{H}_{\eta(x_j)})|
		\\&\geq  \sum_{j=1}^k\frac{1}{2}|  N^u_{x_j}X|=\frac{1}{2}
		|\bigcup_{j=1}^k N^u_{x_j} X  |=\frac{1}{2}  |N^uX| \\
		&=\frac{1}{2}|\mathbb{S}^{N-1}(\rho)|.
		\end{align*}
\end{proof}                               

Then we can finish the proof of Proposition \ref{prop1.4}.

\

\noindent{\it \textbf{Proof of Proposition \ref{prop1.4}}.} Now one can follow closely the ideas given in \cite{CGR06} to finish the  proof of  the Proposition.
For the convenience of the reader, we sketch the ideas of proof. As above, we consider the graph $\tilde X$ of $X$ and $N^{\tilde \rho} \tilde X$.
First, since $u$ is continuous,  $\tilde X$ is also compact. One can  show that $N \tilde X /\bar \sigma$ is  closed and 
hence  $N^uX/\sigma$ is also
closed.
Then,  for any integer $i$, $\tilde X$ is covered by finitely many balls in $\R^{N+1}$ of
radius $1/i$ centered at points of $\tilde X$. Let $\tilde X_i$ be the set of the centers and $X_i\subset X$ its projection into $\R^N$ by forgetting the last coordinate 
$\xi_{N+1}$. 
It is clear that $\tilde X_i$ converges to $\tilde X$  ($X_i$ converges to $X$ resp.) in the Hausdorff distance sense.
In view of the remark that $\tilde X$ is also disjoint from the interior of its convex hull, we can apply the same proof as in \cite{CGR06} to conclude
that $N_{\tilde p} \tilde X_i$ converges to $N_{\tilde p} \tilde X$ in the Hausdorff distance sense,
for any $\tilde p\in \tilde X$. %Denote  the projection of $\tilde X_i$ 
It follows that $N^uX_i$ converges to $N^uX$, for $N^uX_i = N \tilde X_i\cap \{\xi_{N+1}=-1\}$ by Lemma \ref{lem2}.
Since $\sigma$ is continuous, it follows that $N^uX_i/\sigma$
converges to $N^uX/\sigma$ in the Hausdorff distance sense. Finally,  in view of Lemma \ref{lem3} and  the fact that 
$N^uX/\sigma$ is closed, the contradiction argument as in \cite{CGR06} completes the proof.

\qed

 \section{A new proof of the relative isoperimetric inequality} \label{sec2}

 In this section, in order to well present
our methods and ideas,  we give a new proof for the relative isoperimetric inequality of domains,  Theorem 
\ref{thm1.4}, by using Proposition \ref{prop1.4} and the ABP method.
The original proof given by Choe-Ghomi-Ritor\'e in \cite{CGR07} uses Proposition \ref{prop1.1} proved in \cite{CGR06} and the minimization of the relative isoperimetric domain.

 As mentioned in the Introduction, Cabr\'e gave a simply proof of the classical \iqq, by using 
the technique introduced by Alexandrov, Bakelman, and Pucci to establish the ABP estimate. 
%One could naturally ask if there is a similar proof for the \riq of Choe-Ghomi-Ritor\'e \cite{CGR07}. 
%With this aim in mind: 
We generalize Cabr\'e's idea to provide 
a new proof of the relative isoperimetric inequality. Namely we provide a boundary version of his proof. 

%However, the boundary causes  some  big problems.
%We modify the ideas in another paper of Choe-Ghomi-Ritor\'e  \cite{CGR06} 
%to handle these extra difficulties arising from the boundary.  
%It is interesting to see that our proof can be used to generalize the result 
%of Brendle's isoperimetric inequality for minimal submanifold in the next Section.

%One can  see in this section that this boundary version is complete non-trivial. The crucial part is a modification of the 

Let $\mathcal{C}\subset \mathbb{R}^{N}$ be an open convex body in $\mathbb{R}^{N}$ with a smooth boundary
 $S=\p \C$. Let $M:=\Omega \subset  \RR^ N$ be a bounded domain outside $\C$ with its 
boundary $\P \Omega$ consisting of two smooth pieces $\Sigma$ and $\Gamma$, 
where $\Gamma\subset \p \C$. Both $\Sigma$ and $\Gamma$ can be non-connected and $\Gamma$ is closed and hence compact. 
Their common boundary is denoted by $\p \S$. Let $\nu$ be the unit outward normal vector field of $\p \Omega$
 and $\nu_S$ be  the unit outward normal vector field of $S$. It is clear that $\nu=-\nu_S$ along $\Gamma$. (Note that in this Section
 we consider the case of codimension $0$.
 In the higher codimensional case, $\nu=-\nu_S$ along $\Gamma$ is the free boundary condition.)
 
We assume first that 
\begin{equation}\label{orth}  \Sigma \hbox{ intersects } S \hbox{ orthogonally.}
 \end{equation}
 This is equivalent to that  $\Sigma \hbox{ intersects } \Gamma \hbox{ orthogonally.}$ (For the general case, we will use a simply approximation argument to reduce to this 
 case. See the proof below.)
Under this assumption we consider the following problem
\begin{eqnarray}\label{3.2a}
 \Delta u &=& \frac {|\Sigma|}{|\Omega|},  \quad \hbox { in }\Omega,\\
 \frac {\p u}{\p \nu}&=& 1, \quad \hbox { in }\Sigma,\\
 \frac {\p u}{\p \nu}&=& 0, \quad \hbox { in }\Gamma\backslash \p \Sigma.\label{3.4b}
\end{eqnarray}
The existence of a weak solution is easy to show. Due to the Neumann condition on $\Gamma$ and the orthogonality in \eqref{orth} one can show that
$u\in C^{1,\a}(\overline \O)\cap C^\infty_{loc}(\overline \O \backslash \p\Sigma)$ for some $\alpha\in (0,1)$.

\begin{rem} It is easy to give a weak formulation to problem \eqref{3.2a}-\eqref{3.4b} and obtain its
weak solution $u$. By a standard elliptic method, one can show that $u\in C^\infty_{loc}(\overline{\O}\setminus\p \S)$. A
regularity problem of $u$ might occur along the “corner” $\p \S$. However, due to \eqref{orth} a
reflection argument given for instance in \cite{GJ} provides a proof for $u\in C^1(\overline\O)$. In fact,
after a reflection along $\Gamma$, we obtain a domain with a $C^1$ boundary $\tilde \Sigma$ which contains $\p \S$
in its interior and a small portion of $\Sigma$ and its reflection. $u$ can be also reflected so that
the resulting $u$ satisfies an elliptic equation with $C^1$ coefficients weakly. Moreover, on $\tilde\Sigma$ 
we have continuous oblique boundary conditions so that we can use the results in \cite[Section 4.1]{Lie} to get the $C^1$ estimate. This approach works also for problem  \eqref{3.3}-\eqref{3.5}
below.
\end{rem}
Without loss of generality, by scaling, we may assume that\[ \frac {|\Sigma|}{|\O|}=N.\]
Now we define its lower contact set, as in the  ABP method, by
\[\Gamma_+ :=\{x\in \Omega \,|\,  u(y) \geq  u(x) + \langle \n u(x) , y - x\rangle  \quad \forall \,  y \in \Omega\}.\]
If we can prove that
\begin{equation}\label{eq1.1}
 | \n u(\Gamma_+) | \ge \frac 12  |\mathbb{B}^N|,
\end{equation}
then we can follow the proof of Cabr\'e \cite{Cabre2008}. See at the end of this section.
One might hope that $$\n u(\Gamma_+) \hbox{ contains a half unit ball,}$$
which obviously implies \eqref{eq1.1}. Unfortunately, this is in general not true. 
To overcome this difficulty, for any $\rho\in (0, \infty)$, we consider following sets
\[
 \Gamma^\rho_+ :=\{x\in \Gamma_+\,|\, |\n u(x)|<\rho\} \quad \hbox{ and } \quad  \p\Gamma^\rho_+ :=\{x\in \Gamma_+\,|\, |\n u(x)|=\rho\}.
\]
We want to prove 
\begin{equation}\label{eq1.2}
 |\n u (\p\Gamma^\rho_+)| \ge \frac 12 |\SS^{N-1}(\rho)|, \qquad \qquad\forall \rho\in (0,1).
\end{equation}
Then \eqref{eq1.1}  clearly follows from  \eqref{eq1.2}.  We use Proposition \ref{prop1.4} to prove the area estimate \eqref{eq1.2}.
%Our idea  to prove \eqref{eq1.2} is motivated by
%a new concept, {\it the restricted normal cone}, which is introduced by  Choe-Ghomi-Ritor\'e in \cite{CGR06}.

\begin{prop}\label{slice area estimate for domains}
 \eqref{eq1.2} is true, namely \[
 |\n u (\p\Gamma^\rho_+)| \ge \frac 12 |\SS^{N-1}(\rho)|, \qquad \qquad\forall \rho\in (0,1).\]
 It follows that 
 \[|\n u (\Gamma^1_+)| \ge \frac 12 |\mathbb{B}^N|, \]
 and hence \eqref{eq1.1} is true.
\end{prop}

\begin{proof} Let $\rho\in (0,1)$ be fixed. 
Let $X:=\Gamma$. Since $S$ is by assumption convex, $\Gamma$ is disjoint from the interior of its convex hull.
Let $u:\Gamma \to\R$ be the restriction of $u$ and $\sigma:\Gamma\to \SS^{N-1}$  be the outer unit normal $\nu_S$ of $S$ along $\Gamma$. 
We first check that
$\bar \sigma:=\tilde \rho (\nu_S, 0)\in \SS^{N}(\tilde \rho)$ is an element of $N_{\tilde p} \tilde \Gamma$ for any $p\in \Gamma$. That is, we need to show
that
\[
 \<\tilde x-\tilde p,\bar \sigma (p)\> \le 0, \quad \forall x \in \Gamma.
\]
This is certainly equivalent to
\[
 \<x- p, \nu _S(p)  \> \le 0, \quad \forall x \in \Gamma,
\]
which is true, due to the convexity of $S$. Hence from  Proposition \ref{prop1.4} we have
\begin{equation}
 \label{a1}
 |N^u\Gamma/\nu_S| \ge \frac 12 |\mathbb{S}^{N-1}(\rho)|, \qquad \forall \rho\in (0,1).
\end{equation}
Now we \textbf{claim} that
\begin{equation}\label{a2}
 \n u (\p \Gamma_+^\rho)\supset N^u\Gamma /\nu_S.
\end{equation}
Then the first statement of the Proposition follows clearly from this claim and \eqref{a1}. 
It remains to  prove the claim.
For any $\xi\in 	N_p^u\Gamma/\nu_S $, we have by definition%i.e. there exists some point $p\in\Sigma$ such that $\xi\in N_p^\rho\Sigma/\eta$, then it holds that $\xi\in N_p^\rho\Sigma$ 
	\begin{eqnarray}\label{b1}
	\langle x-p,\xi\rangle  &\leq&   u(x)-u(p),  \qquad \forall x\in\Gamma,\\
	  	\langle \xi,\nu_S(p)\rangle &\geq & 0. \label{b2}
	\end{eqnarray}	
	Define a function $f:\Omega \to \R$ by
		$$f(x)=u(x)-\<x,\xi\>.$$
			\eqref{b1} means that  $p$ is a minimum point of $f$  on $\Gamma$. We have two cases:
			either i) $p$ is a   minimum point of $f$  in the whole $\overline\O$, or ii) 
			$p$ is not a   minimum point of $f$  in the whole $\overline \O$.
			
			We first consider case ii). In this case, there exists another point $q\in \overline \O \backslash \Gamma$ such that $f(q)=\min\limits_{x\in \overline \O} f(x)$.
			If $q\in \S$, then by the definition of $u$ we have
			\[\frac {\p f} {\p\nu} (q) =\frac {\p u} {\p\nu} (q)-\<\nu(q), \xi\> \ge  1- \rho >0,\]
			for $|\xi|=\rho.$ This is impossible. Hence $q\in \O$ and $\n f(q)=0$,
			which implies that $\xi =\n u (q)$. Since $q$ is a minimum point of $f$ in $\overline \O$,
			it is  easy to see that $q\in \Gamma_+$, and hence $\xi\in \n u (\partial \Gamma^\rho_+).$
			
	Now we consider case i). In this case, we know all tangential derivatives of $f$ along $\Gamma$ vanish and $\frac {\p f}{\p \nu} (p) \le 0.$	However, 
by	using  \eqref{b2} and %the free boundary condition, i.e.,
 $\nu=-\nu_S$ along $\Gamma$, it yields that
	\[
	0\geq  \frac {\p f}{\p \nu}(p)=  \frac {\p u}{\p \nu}(p)-\<\nu(p), \xi\>=\<\nu_S(p), \xi\> \ge 0.
	\]Hence $\frac {\p f}{\p \nu}(p)= 0$, and hence $\n f (p)= 0$, which implies that $\xi=\n u (p)$.
	The minimality of $p$ then implies that $p\in \Gamma_+$. It follows that $\xi \in \n u (\p\Gamma_+^\rho)$. The claim holds.

	The second statement follows from
\[
 |\n u (\Gamma^1_+)|=\int_0^1 |\n u (\p\Gamma^\rho_+)| d\rho \ge \frac 12 |\SS^{N-1}| \int_0^1 \rho^{N-1}d\rho 
 =\frac 12  \frac 1N|\SS^{N-1}| = \frac 12 |\mathbb{B}^N|.
\]
	
\end{proof}
Now we can finish the proof of the \riq of Choe-Ghomi-Ritor\'e \cite{CGR07}, Theorem \ref{thm1.4}.

\

\noindent{\it \textbf{Proof of Theorem \ref{thm1.4}}.} If $\O$ satisfies Assumption \eqref{orth}, we consider the function $u$ defined by (2.1)-(2.3).
From the above discussions, we have
\begin{eqnarray*}
 \frac 12  |\mathbb{B}^N| &\le & | \n u(\Gamma^1_+) | \le \int _{\n u(\Gamma^1_+)} dx  \\
                 & \le  &  \int _{\Gamma^1_+}  \det \n^2 u(x) dx \le \int _{\Gamma^1_+}  \left(\frac {\Delta u}N \right)^N  dx \\
                 &\le& |\O |= \frac 1 {N^N} \left(\frac {|\Sigma|^N}{|\Omega|^{N-1}}  \right),
\end{eqnarray*}
recalling that $\frac{|\S| }{|\O|}=N$.
This is the optimal \riqq. 

If  $\O$ does not satisfy Assumption \eqref{orth}, one can use an approximation argument. It is not difficult to see that for any $\epsilon>0$, one can construct a domain 
$\O_\epsilon$ as above satisfying  Assumption \eqref{orth} such that the difference between the volumes of $\O$ and $\O_\epsilon$
and the difference between the areas  of their relative boundaries are smaller than $\epsilon$. The isoperimetric inequality holds for $\O_\epsilon$
, which implies the isoperimetric inequality
for $\O$.

Now we consider the equality case. Assume that $\O$ with boundary $\Sigma$ and $\Gamma$  achieves the equality.
Such a domain is called a  relative isoperimetric domain. By the first variational formulas for the area and the volume, it is easy to
see that the relative boundary $\Sigma$ intersects the support surface  $S$ orthogonally, i.e., Assumption \eqref{orth} holds true. For a proof see 
\cite{RV} or the Appendix.
Hence we can define $u$ the  solution  of (2.1)-(2.3)
and carry on the argument presented above to obtain the above inequality. Now by the assumption that $\O$ achieves in fact equality, we have 
$|\Gamma^1_+|=|\O|$ and 
\[\n^2 u=  I \quad  \hbox { on }\Gamma^1_+,\]
where $I$  is the identity map. Let $x_0$ be a minimum point of $u$ in $\overline\O$. By the definition of $u$, we know that $x_0$ can not be on $\Sigma$. Hence, 
either $x_0\in \Omega$ or $x_0\in \Gamma$. In the both cases, we have $\n u(x_0)=0$. Without loss of generality, assume that $x_0=0$ and $u(0)=0$. Then it follows that
\[u(x)=\frac 12 |x|^2.\]
 Now it is easy to see that $\O \subset \mathbb{B}^N$ and $\S\subset \SS^{N-1}$, for $\frac {\p u}{\p \nu} (x)=1$, 
for any $x\in \Sigma$. Since the origin $0$ is either  outside of the convex body or on its boundary $S$, 
there exists a hyperplane through the origin $0$, which does not intersect the interior of the convex body. It divides the unit ball into two half balls.
It follows that one of the half balls is contained entirely in $\O$. Since the volume of $\O$ is the same as the volume of a  unit half  ball, $\O$ must be the
unit half ball. Hence we finish the proof. \qed

 \section{\riq for minimal submanifolds} \label{sec3}
 
 In this Section, we consider the higher codimensional cases and prove Theorem \ref{thm1.2}, the \riq for submanifolds in the Euclidean space.
 
Let $\mathcal{C}\subset \mathbb{R}^{N}$ be  an open convex body in $\mathbb{R}^{N}$ with a smooth boundary
 $S=\p \C$. %In this paper we always assume that $\mathcal{C}$ has  interior points and does not equals 
 %to $\mathbb{R}^{N}$.
 %Here our convex body means that it is a compact convex set with interior points. 
 Let $M \subset \RR^N$ be an $n$-dimensional submanifold with codimension $m=N-n$. Its boundary $\p M$
 consists two smooth pieces $\Sigma$ and $\Gamma$, where $\Gamma\subset \p \C$ and closed. Denote their common boundary by $\p\Sigma$, which may be empty.
 %Both $\Sigma$ and $\Gamma$ may be or may not be conn
 Let $\nu$ be the outer unit normal vector field of $\p M \subset M$
 and $\nu_S$ be  the outer unit normal vector field of $S$. We assume that  $M$ is a partially free boundary submanifold 
 with free boundary $\Gamma$ on the support $S$, i.e.,
 $\nu=-\nu_S$ along $\Gamma$.
 
 First, by  scaling we may assume that
 \begin{align}\label{relative comptiable condition}
|\Sigma|+\int_M |H|dv=n |M|.
 \end{align}
 As in Section \ref{sec2},
 we first consider the case \begin{align}\label{orthogonal 2}
  \Gamma \text{ meets }\Sigma \text{ orthogonally along } \Sigma\cap \Gamma, 
 \end{align}
 and the following problem %\textcolor{blue}{then from elliptic theory, (see for example, \cite{DLK2018}, Theorem 1.7)}  there exists $u:M\to\mathbb{R}$ with $u\in C^{1}(\overline{M})\cap C^{2}(M)$ solving equations
 \begin{eqnarray}\label{3.3}
 \Delta u &=& n-|H|,  \quad \hbox { in } M,\\\label{3.4}
 \frac {\p u}{\p \nu}&=& 1, \quad \hbox { in }\Sigma,\\\label{3.5}
 \frac {\p u}{\p \nu}&=& 0, \quad \hbox { in }\Gamma\setminus \partial\Sigma,
 \end{eqnarray}
 where $\nu$ is the unit outward normal vector field of $\partial M$ in $M$. As above we can show that there exists a solution
 $u\in C^{1,\alpha}(\overline M)\cap C^2_{loc}(\overline M \backslash \p \Sigma)$ solving equations \eqref{3.3}-\eqref{3.5} for some $\alpha\in (0,1)$.
 
 For any $x\in M$, let $T_xM$  and $T^\perp_x M$ be the tangential space and normal space of $M$ at $x$ respectively.
 Let $\Pi$ be the second fundamental form  of $M$, which is defined by
 $\<\Pi (X,Y), V\> =\<\bar D_XY, V\>$, for any $X, Y \in T M$ and $V\in T^\perp M.$ Here $\bar D$  
 is  the standard connection  in $\R^N$. We use $\n$ to denote the connection on $M$ w.r.t. the induced metric $g$.
 
 Following Brendle \cite{Brendle2019} we define
 \begin{eqnarray*}
 & &U:=\{x\in M\setminus \Sigma|\quad |\nabla u|(x)<1\}\subset M,\\
 &&\Omega:=\{(x,y)\in (M\setminus  {\Sigma})\times T_x^\perp M|\quad  |\nabla u|^2(x)+|y|^2< 1\},%\subset M\setminus\Gamma\times T^\perp M,
 \\
& & A:=\{(x,y)\in  \O|  {\nabla^2 u(x)-\langle \Pi_x,y\rangle\geq 0}\},%\label{slice level set3} %\subset\Omega_\rho,
 \end{eqnarray*}
 and 
 \begin{eqnarray*}
 \Phi:\quad  \Omega &\to & \mathbb{R}^{N},\\
 (x,y)&\mapsto& \nabla u(x)+y.
 \end{eqnarray*}
 It is clear that  $|\Phi(x,y)|^2=|\nabla u|^2(x)+|y|^2.$ The following statements was proved in \cite[Lemma 5 and Lemma 6]{Brendle2019}.

 \begin{lem}\label{lem3.1}
 \
 %For any $(x,y)\in \O$, the map $\Phi$ satisfies that 
 	\begin{enumerate}
 		\item  	For any $(x,y)\in \O$, the Jacobian determinant of $\Phi$ satisfies
 		$
 	\det	( \hbox{\rm Jac}\, \Phi)(x,y)=\det\big( (\nabla^2 u)(x)-\langle \Pi_x,y\rangle\big).
 		$
 		\item  For any $(x,y)\in A$, the Jacobian determinant of $\Phi$ satisfies
 		$
 		0\leq \det(\mbox{\rm Jac}\, \Phi)(x,y)\leq 1.
 		$
 		%For any $(x,y)\in \O$,
 		In particular, if $\det(\mbox{\rm Jac} \,  \Phi)(x,y)= 1$ at point $(x,y)$, then 
 		$\n^2 u(x)-\langle \Pi_x,y\rangle=g_x$, where $g_x$ is the induced metric $g$ at $x$.% on $M$ evaluated at $x$.
 	\end{enumerate}
 \end{lem}
 As in the codimension $0$ case, we can not hope that $\Phi(A)$ contains a half unit ball.
 For our use, we set
 \begin{eqnarray*}
 %\Omega_\rho&:=&\{(x,y)\in (M\setminus  {\Sigma})\times T_x^\perp M: |\nabla u|^2(x)+|y|^2\leq \rho^2\},\\
%A_\rho&:=&\{(x,y) \in \O_\rho :
%{\nabla^2 u(x)-\langle \Pi_x,y\rangle\geq 0}\},\\
\p A_\rho&:=& \{(x,y) \in \O : |\nabla u|^2(x)+|y|^2= \rho^2 \, \hbox{ and } \,
{\nabla^2 u(x)-\langle \Pi_x,y\rangle\geq 0}\}.
\end{eqnarray*}
It is clear that $A=\cup_{\rho \in (0,1)} \p A_\rho.$

 Now we prove the following Proposition by using the ideas given in the previous Section.
 
 \begin{prop}\label{slice area estimate} For any $\rho\in (0,1)$, there holds
 	\begin{align*}
|\Phi(\p A_\rho)|\geq \frac{1}{2}	|\mathbb{S}^{N-1}(\rho)|.
 	\end{align*}
 \end{prop}
 \begin{proof}
  Let us consider the function $u$ defined by \eqref{3.3}-\eqref{3.5}
  and define the generalized normal cone $N^u \Gamma$ of length $\rho\in (0,1)$ as in the previous section.
  Due to the free boundary condition and the convexity of the support hypersurface $S$, one can check easily as in the previous Section
that Proposition \ref{prop1.4}
  can be applied to  our current case. Hence we have
  \[ |N^u \Gamma/\nu_S|\ge \frac 12 |\SS^{N-1}(\rho)|.
   \]
Therefore, the Proposition follows from the next Lemma.
 \end{proof}

 \begin{lem}\label{slice inclusion the cone}For any $\rho\in(0,1)$, there holds
 	\begin{align*}
 	\Phi( \p  A_\rho)\supset N^u\Gamma/\nu_S.
 	\end{align*}
 \end{lem}
 
 \begin{proof} For any $\xi\in 	N^u_p\Gamma/\nu_S$,  it holds by definition  that 
 	\begin{eqnarray}\label{eq3.6}
 	\langle x-p,\xi\rangle &\leq&   u(x)-u(p), \forall x\in\Gamma,\\
     \langle \xi,\nu_S(p)\rangle &\geq & 0. \label{eq3.7}
 	\end{eqnarray}	
 	Again we consider a function defined
 	\begin{eqnarray*}
 	f:\quad \overline M &\to& \mathbb{R}\\ 
 	x &\mapsto & u(x)-\langle \xi,x\rangle.
 	\end{eqnarray*}
 	First, notice  that this function satisfies 
 	\begin{align}\label{normal derivative on Sigma}
 	\frac {\p f}{\p \nu}(x) =\frac {\p u}{\p \nu}(x)
 	-\langle \xi,\nu(x)\rangle. 
 	\end{align}
 	Hence we have $\frac {\p f}{\p \nu}(x)=1-\langle \xi,\nu(x)\rangle >0$ for any $x\in \Sigma$, for $|\xi|=\rho<1$.
 	This means that the function $f$ can not achieve its minimum on $\Sigma$. \eqref{eq3.6} means that $p$ is a minimum point of
 	$f|_{_{\Gamma}}:\Gamma\to \R$.
 Hence we have only two cases:
 	either $f: M\to \R$  achieves its minimum  at $p$,  or, at another point  $q\in M$. 
 	
 	In the latter case 
 	$q $ is an interior minimum point of $f$. Thus we have
 		\begin{eqnarray} \label{eq3.9}
 		0&= &\nabla f(q)=\nabla u(q)-\xi^T,\\
 		\label{eq3.10}
 		0& \leq &\nabla^2 f(q)=\nabla^2u(q)-\langle \xi^\perp, \Pi_{q}\rangle,
 		\end{eqnarray} 
 		where $\xi^\perp$ is the normal part of $\xi$ in $T^\perp M$. 
 			(We remark that here is one of the differences between the higher codimensional case
 				and the $0$ codimension case.)   Set $\xi^T:=\xi-\xi^\perp \in TM$.
 		It implies % if we use the decomposition of the tangential part and the normal part of $\xi$,
 		\begin{align*}
 		\xi=\xi^T+\xi^\perp= \nabla u(q)+\xi^\perp=\Phi(q,y_0),
 		\end{align*}for  $y_0:=\xi^\perp\in T^\perp_q M$. 
 		%Note that $\langle \xi,\sigma_q\rangle=\langle \xi^T+y_0,\sigma_q\rangle =\langle y_0,\sigma_q\rangle$.
 		In particular, it holds that
 		\begin{align*}
 		|\nabla u(q)|^2+|y_0|^2=|\xi|^2=\rho^2<1.
 		\end{align*}
 		That is,  $\Phi(q,y_0)=\xi$ and $(q,y_0)\in \p A_\rho$.
 		
 	Now we consider the first case, i.e., $p$ is a minimum point of $f$ in $M$. Due to the Neumann boundary condition \eqref{3.5} and \eqref{eq3.7},
 	we have 
 	$$ \frac {\p f}{\p \nu}(p) =\frac {\p u}{\p \nu}(p)
 	-\langle \xi,\nu(p)\rangle= -\langle \xi,\nu(p)\rangle=  \langle \xi,\nu_S(p)\> \ge 0.$$
 	This implies, together the minimality of $p$, that $\n f(p)=0$. From this, one can show that $\n^2 f (p) \ge 0$, though $p$ is a boundary point.
 	Both together mean that equation \eqref{eq3.9} and  \eqref{eq3.10} hold at $p\in \Gamma$. Then the same argument given above implies that
 	$\xi=\n u (p)+\xi^\perp \in \Phi(\p A_\rho)$. 
 \end{proof}

 Now we are ready to prove one of our main results. 
 
 \begin{proof}[\textbf{Proof of Theorem \ref{thm1.2}}] %Assume right now that $k=N-1$, i.e. for codimension 1 case. 
 We only need to consider the case $m\ge 2$, since the case $m=1$ can be viewed as  the case $m=2$, by embedding $\R^{n+1}$ into $\R^{n+2}.$
 We first assume that Assumption \eqref{orthogonal 2} holds.
In this case, we define $u$ to be a solution of problem \eqref{3.3}-\eqref{3.5}. 
 From the above discussions, we know that Proposition \ref{slice area estimate} holds. Namely, we have
 \begin{align*}%\label{isoeq final 2}
 |\Phi( \p A_\rho)|\geq \frac{1}{2}| \mathbb{S}^{N-1}(\rho)|, \quad\forall \rho\in (0,1).
 	\end{align*}
 	It  yields that  	
 	\begin{align}\label{isoeq final 2}
 	\int_{ \Phi(\cup_{\rho\in(t, 1)}\p A_\rho)} 1d\xi
 	&=\int_t^1 |\Phi(\p A_\rho)|d\rho
 	\\&\geq \int_t^1 \frac 12  |\mathbb{S}^{N-1}|{\rho^{N-1}}d\rho =\frac{|\mathbb{B}^{N}|}{2}(1-t^N). \nonumber
 	\end{align}
 	Now we use a trick of Brendle \cite{Brendle2019}.
 Recall that $\Phi(x,y)=\nabla u(x)+y$ and $|\Phi|^2=|\nabla u|^2+|y|^2$. 
 For  any $0<t<1$, we have
 \begin{equation}\label{isoeq final 1}
 \begin{split}
 	\int_{ \Phi(\cup_{\rho\in(t, 1)}\p A_\rho)}
 	1d\xi
 &= \int_{U}\left(\int_{\{y\in T_x^\perp M: t^2<|\nabla u|^2(x)+|y|^2<1\}} \det (\text{Jac}(\Phi)) \cdot  \chi_A(x,y)dy\right)dv_x
 \\&\leq  \int_{U}\left(\int_{\{y\in T_x^\perp M: t^2<|\nabla  u|^2(x)+|y|^2<1\}} 1dy\right)dv_x
 \\&=|\mathbb{B}^{m}|\int_U\left[(1-|\nabla  u|^2(x))^{\frac{m}{2}}-(t^2-|\nabla  u|^2(x))_+^{\frac{m}{2}}\right]dv_x
 \\&\leq \frac{m}{2} (1-t^2) |\mathbb{B}^{m}| \cdot |M|, 
 \end{split}
 \end{equation}
 	where $\chi_A$ is the characteristic function of $A$, and we have used  $m\geq 2$ in the last inequality.
Combining \eqref{isoeq final 1}, \eqref{isoeq final 2} and dividing $(1-t)$ and letting $t\to 1^-$, we obtain
 	\begin{align*}
 	|M|\geq\frac 12 \frac{N|\mathbb{B}^{N}|}{m|\mathbb{B}^{m}|}.
 	\end{align*}
 	Hence, we have 
 	\begin{align*}
|\Sigma|+\int_M |H|dv=n|M|&=n {|M|^{\frac{n-1}{n}} } \cdot |M|^{\frac{1}{n}}
 	\\&\geq n|M|^{\frac{n-1}{n}}\cdot \left(\frac{N|\mathbb{B}^{N}|}{2m|\mathbb{B}^{m}|}\right)^{\frac{1}{n}},
 	\end{align*}
 	which yields that	\begin{align*}
 	\frac{| \Sigma|+\int_M |H| dv}{ |\partial \mathbb{B}^n| }\geq   \left(\frac{N|\mathbb{B}^{N}|}{2m |\mathbb{B}^n||\mathbb{B}^{m}|}\right)^{\frac{1}{n}} \left(\frac{|M|}{ {|\mathbb{B}^n|} }\right)^{\frac{n-1}{n}}.
 	\end{align*} 	
 One can check easily that when $m=2$, $ {(n+2)|\mathbb{B}^{n+2}|}=2{ |\mathbb{B}^n||\mathbb{B}^2|}$ holds. Hence we have proved the \riq for $m\ge 2$.

 	If 
 	$\S$ and $\Gamma$ are not connected, i.e., their common boundary is an empty set. Then the above proof works without any change.
 	
 	If $\S$ and $\Gamma$ are connected and $M$ does not satisfy Assumption \eqref{orthogonal 2}, one can  use  an
 	approximation argument to construct  domains $M_\varepsilon$ in $M$ with relative boundary  $\Sigma_\varepsilon$ and free boundary $\Gamma_\varepsilon$ on $S=\p \C$  satisfying  \eqref{orthogonal 2} such that 
 	the area and the volume of $\Sigma_\varepsilon$ and $M_\varepsilon$ resp. are  close to the area and the volume of $\Sigma$ and $M$ respectively as small as we want. 
 	Then we obtain \eqref{riq} for $M_\varepsilon$ as above. By taking  $\varepsilon\to 0$, we   obtain \eqref{riq} for a general $M$.
 	
 	Now we consider the case that  $m=2$ and  equality holds.
 	Assume  that $M^n$ is a compact free boundary submanifold in $\mathbb{R}^{n+2}$ such that
 		\begin{equation}\label{riq equality case}
 	\frac{|\S|+\int_M|H|dv}{|\partial \mathbb{B}^n|}
 	=\left(  \frac{(n+2)|\mathbb{B}^{n+2}|}{4 |\mathbb{B}^n||\mathbb{B}^2|}\right)^{\frac{1}{n}} 
 	\left(\frac{\mbox{}|M|}{\mbox{}|\mathbb{B}^n|}\right)^{\frac{n-1}{n}}=\left(\frac{1}{2}\right)^{\frac{1}{n}}\left(\frac{\mbox{}|M|}{\mbox{}|
 	\mathbb{B}^n|}\right)^{\frac{n-1}{n}},
 	\end{equation}
where we have used that $ {(n+2)|\mathbb{B}^{n+2}|}=2{ |\mathbb{B}^n||\mathbb{B}^2|}$ in the last equality. 
Such a submanifold $M$ is called a \textit{relative isoperimetric region}, which is by the \riq \eqref{riq}
a stationary point of functional \eqref{functional} below.
It can be proved that Assumption \eqref{orthogonal 2} holds for $M$, namely the relative boundary $\S$ interests $S$ orthogonally. See Proposition \ref{variational character} in the Appendix.
Hence we can follow the above argument to define a function $u$ and obtain 
the isoperimetric inequality. Since for $M$ we have equality, all inequalities in the above proof are equalities.
In particular, we conclude that $\det(\text{Jac}\Phi) \cdot \chi_A=1$ a.e. in $M$ and $|U|=|M|$. From Lemma \ref{lem3.1}, we obtain that 
$\nabla^2 u(x)-\langle \Pi_x,y\rangle =g_x$ for a.e. in $\Omega$. Since $u\in C^2(M)$, we have 
$\nabla^2 u(x)-\langle \Pi_x,y\rangle =g_x$ for all $(x,y)\in\O$.  Since  $g_x$ and $y$ are independent, it follows that $\nabla^2 u(x)=g_x$ for all $x\in M$ and $\Pi\equiv 0$ on $M$, 
and  $M$ is contained in an $n$-dimensional flat space $P=\mathbb{R}^n$. It is clear that we are now in the case of codimension $0$
and the argument given at the end of  last section implies that $M$ is a flat half $n$-ball. Hence we complete the proof.
 \end{proof}
As a direct consequence, we have a boundary version of 
Michael-Simon and  Allard inequality  \cite{MS}, \cite{Allard}  with an optimal positive constant, at least in the case of codimension $m\le 2$. 
 
 \begin{thm}\label{coro2}
 	Let $M^n\subset \R^{n+m}$ $(m\ge 1)$ be a partially free  boundary   submanifold 
with relative boundary  $\Sigma$ and free boundary $\Gamma$ on a convex support hypersurface $S$.  For any non-negative smooth function $f:M \to \R$ vanishing on the relative boundary
$\Sigma$, we have
	\begin{equation}\label{riq10}
 	\frac{\int_M (|\n f| +|H|f)dv}{|\partial \mathbb{B}^n|}\geq 
 	\left(\frac 12 \right)^{\frac{1}{n}} b_{n,m}
 	\left(\frac{\mbox{}\int_M f^{\frac n{n-1}}dv}{\mbox{}|\mathbb{B}^n|}\right)^{\frac{n-1}{n}}.
 	\end{equation}
 	\end{thm}
 
 \begin{proof} 
  Since  the superlevel set $\{f \ge s\}$ is a partially free boundary submanifold, one  can 
   apply Theorem \ref{thm1.2} to $\{f \ge s\}$ and then follow
  completely the proof of Brendle in \cite{Brendle2019}.
 \end{proof}

 \section{Appendix} \label{sec5}

 \begin{prop}\label{variational character}
  Under the assumptions in Theorem \ref{thm1.2}, if $M$ achieves equality in the \riq \eqref{riq}, the relative boundary $\S$ and the free boundary $\Gamma$ of $M$ are connected, then
  Assumption \eqref{orthogonal 2} holds true.
  Moreover, its relative boundary $\S$ satisfies
  \[
   h- |H|_{|\S} =const.,
  \]
  where $h$ is the mean curvature of $\Sigma \subset M$.
 \end{prop}
 \begin{proof} First of all, we extend $M$ smoothly to  a partially free boundary submanifold  $\tilde M$  in $\R^N$ with free boundary $\tilde \Gamma$ on $S$
 and denote $\tilde H$ be the mean curvature vector field of $\tilde M$ in $\R^N$.
 Recall that $\Sigma$ and $\Gamma$ are the relative boundary and free boundary of $M$ resp. with a non-empty common boundary denoted by $\p \S$.
 Then we consider a variation of $\Sigma$ in $\tilde M$, namely, $F:(-\e, \e)\times \Sigma \to \tilde M$ such that 
 $F_t:\Sigma \to \tilde M$, ($t\in (-\e, \e)$), defined by $F_t(x)=F(t,x)$ is an immersion with $F_t(\p\S)\subset \tilde \Gamma$ and $F_0=id.$
 Let $M_t$ be the domain in $\tilde M$ enclosed by $F_t(\S)$ and $\tilde \Gamma$. 
 It is clear that $M_t$ is a partially free boundary submanifold with relative boundary $\S_t:=F_t(\p \S)$ and free boundary on
 the support hypersurface $S$. %\tilde \Gamma$.
 %$F_t(\Gamma)$ is a partially free boundary immersed submanifold with free boundary $F_t(\Gamma)$ and relative boundary $\p \S_t:= F_t(M)$.
 Define a functional by 
   \begin{align}\label{functional}
 	J(M_t):= \frac{|\Sigma_t|+\int_{M_t} |\tilde H| dv}{ |M_t|^{\frac{n-1}{n}}},
 	\end{align}
 	where $\tilde H$ is the mean curvature vector field of $M_t$ in $\R^N$. Note that $\tilde H$ is just the restriction of 
 	the mean curvature of $\tilde M$ in $\R^N$. By Theorem \ref{thm1.2} $M$ satisfies the \riqq. Hence $J$ achieves
 	its minimum at $t=0$, which implies $\frac{ d }{dt}\big|_{t=0} J(M_t)=0.$
 Recall that $\nu$ is the outer unit normal of $\Sigma$ in $M$ and denote
 $\eta$ be the unit outer conormal of $\p \S$ in $\S$. A direct computation gives 
 \begin{align*}
 	0&
 	=\frac{d}{dt}\Big|_{t=0} J(M_t) \cdot |M|^{\frac {n-1}n}
 	\\&=  \int_{\partial \Sigma} \langle X ,\eta\rangle ds - \int_\Sigma h \langle X, \nu \rangle  
 	+\int_\Sigma |H|_{|_\Sigma}\langle X ,\nu\rangle 
 	-\frac{n-1}{n|M|}\left(|\Sigma|+\int_M |H| dv\right)\cdot \int_\Sigma \langle X,\nu\rangle,
 	\end{align*}
 	where $X$ is the variation vector of $F$ defined by
 	\[X(x)=\frac{\p F}{\p t}\Big|_{t=0} (x), \quad\forall x\in\Sigma.\]
 	%Now the Proposition follows easily.
 	From this formula we first get  $h-|H|_{|\S} =const.,$ by considering normal variations $X=\phi\nu$ with support not touching $\p \S$.
 	It follows that $0= \int_{\partial \Sigma} \langle X ,\eta\rangle ds$. Now by considering the variations keeping the property that
 	$F(\p \S)\subset \tilde \Gamma$, we have that  $\eta$ is orthogonal to $\Gamma$, Assumption \eqref{orthogonal 2}. Hence we have completed the proof.
 \end{proof}

\noindent {\it \textbf{Acknowledgement}.}
This work was supported by SPP 2026 of DFG ``Geometry at infinity'' and NSFC (Grant No. 12101255, 12201003). We would like to thank the referee for careful  reading and valuable suggestions to improve the context of the paper. 
\

\end{document}